\documentclass{amsart}
\usepackage{amssymb,mathtools}
\usepackage[british]{babel}
\usepackage{enumitem}
\usepackage[latin1]{inputenc}
\usepackage{xcolor}
\usepackage{url}

\newtheorem{theorem}{Theorem}
\numberwithin{theorem}{section}
\newtheorem{corollary}[theorem]{Corollary}
\newtheorem{lemma}[theorem]{Lemma}
\newtheorem{proposition}[theorem]{Proposition}

\theoremstyle{definition}
\newtheorem{definition}[theorem]{Definition}

\title{Provable better quasi orders}
\author[A.~Freund]{Anton Freund}
\author[A.~Marcone]{Alberto Marcone}
\author[F.~Pakhomov]{Fedor Pakhomov}
\author[G.~Sold\`a]{Giovanni Sold\`a}

\address{Anton Freund, University of W\"urzburg, Institute of Mathematics, Emil-Fischer-Stra{\ss}e~40, 97074 W\"urzburg, Germany}
\email{anton.freund@uni-wuerzburg.de}
\address{Alberto Marcone, Dipartimento di Scienze Matematiche, Informatiche e Fisiche, Universit\`a di Udine, Via delle Scienze, 208 -- Loc.\ Rizzi, 33100 Udine, Italy }
\email{alberto.marcone@uniud.it}
\address{Fedor Pakhomov and Giovanni Sold\`a, Department of Mathematics: Analysis, Logic and Discrete Mathematics, Ghent University, Krijgslaan 281 S8, 9000 Ghent, Belgium}
\email{fedor.pakhomov@ugent.be{\normalfont, }giovanni.a.solda@gmail.com}

\thanks{The work of Anton Freund has been funded by the Deutsche Forschungsgemeinschaft (DFG, German Research Foundation) -- Project number 460597863. Alberto Marcone's research was partially supported by the Italian PRIN 2017 Grant \lq\lq Mathematical Logic: models, sets, computability\rq\rq. The work of Fedor Pakhomov and Giovanni Sold\`a has been funded by the FWO grant G0F8421N}

\begin{document}

\begin{abstract}
It has recently been shown that fairly strong axiom systems such as $\mathsf{ACA}_0$ cannot prove that the antichain with three elements is a better quasi order~($\mathsf{bqo}$). In the present paper, we give a complete characterization of the finite partial orders that are provably~$\mathsf{bqo}$ in such axiom systems. The result will also be extended to infinite orders. As an application, we derive that a version of the minimal bad array lemma is weak over~$\mathsf{ACA_0}$. In sharp contrast, a recent result shows that the same version is equivalent to \mbox{$\Pi^1_2$-}comprehension over the stronger base theory~$\mathsf{ATR}_0$.
\end{abstract}


\maketitle

\section{Introduction}

Let us first recall the notion of better quasi order. We write $[X]^{<\omega}$ and $[X]^\omega$ for the collection of finite and countably infinite subsets of a set~$X$, respectively. When we have $X\subseteq\mathbb N$, we identify elements of $[X]^{<\omega}\cup[X]^\omega$ with the strictly increasing sequences that enumerate them. For such sequences, we write $s\sqsubset t$ to express that $s$ is a proper initial segment of~$t$ (where $t$ may be infinite but $s$ is necessarily finite). Even when we consider them as sequences, we use notation such as $\subseteq$ for the usual notions on sets. By $\subset$ we denote the proper subset relation.

A block is a set $B\subseteq[\mathbb N]^{<\omega}$ such that the so-called base $\bigcup B\subseteq\mathbb N$ is infinite and each $X\in[\bigcup B]^\omega$ admits a unique $s\sqsubset X$ with $s\in B$. It follows that we cannot have $s,t\in B$ with $s\sqsubset t$. A barrier is a block $B$ satisfying the stronger condition that $s\subset t$ holds for no~$s,t\in B$. By a $Q$-array for a quasi order~$Q$, we mean a function $f:B\to Q$ on some barrier~$B$.

For $X\subseteq\mathbb N$ with smallest element~$x$, we put $X^-:=X\backslash\{x\}$. Given $s,t\in[\mathbb N]^{<\omega}$, we write $s\vartriangleleft t$ if there is an $X\in[\mathbb N]^\omega$ with $s\sqsubset X$ and $t\sqsubset X^-$. The latter is decidable since it only depends on~$s\cup t\sqsubset X$. An array $f:B\to Q$ is good if there are $s,t\in B$ with $s\vartriangleleft t$ such that $f(s)\leq f(t)$ holds in the order~$Q$. Otherwise $f$ is bad. A better quasi order ($\mathsf{bqo}$) is a quasi order~$Q$ such that any $Q$-array is good.

Let $[\mathbb N]^1$ be the barrier that contains all singletons. In view of $([\mathbb N]^1,\vartriangleleft)\cong(\mathbb N,<)$, any better quasi order is a well quasi order ($\mathsf{wqo}$). The notion of $\mathsf{bqo}$ has been introduced by C.~Nash-Williams in order to secure stronger closure properties~\mbox{\cite{nash-williams-trees,nash-williams-bqo}}. It plays a crucial role in R.~Laver's famous proof~\cite{laver71} of Fra\"iss\'e's conjecture that the countable linear orders are $\mathsf{wqo}$ under embeddability (which actually shows the stronger result that the $\sigma$-scattered orders are~$\mathsf{bqo}$).

Better quasi orders have been analyzed from the viewpoint of mathematical logic, in particular within the framework of reverse mathematics (see, e.\,g.,~\cite{friedman-rm,simpson09,hirschfeldt-slicing-the-truth,dzhafarov-mummert} for background on the research project and \cite{marcone-survey-new} for a recent survey on $\mathsf{wqo}$s and~$\mathsf{bqo}$s in reverse mathematics). Marcone has shown, in particular, that the notion of~$\mathsf{bqo}$ is $\Pi^1_2$-complete~\cite{marcone-Pi12-complete, marcone-bqoPi12-complete}. We mention the technical but important fact~\cite{cholak-RM-wpo} that the base~$\bigcup B$ of a block~$B$ can be formed in~$\mathsf{RCA}_0$. One obtains an equivalent definition of~$\mathsf{bqo}$ when barriers are replaced by blocks, where the base theory~$\mathsf{WKL}_0$ suffices to secure the equivalence~\cite{cholak-RM-wpo}. Maps $B\to Q$ that are defined on blocks correspond to functions $[\bigcup B]^\omega\to Q$ that are continuous in a suitable sense (see Section~\ref{sect:mba}). As shown by S.~Simpson~\cite{simpson-borel-bqos}, another equivalent definition of~$\mathsf{bqo}$ arises when continuous is weakened to Borel. An intermediate choice yields the notion of $\Delta^0_2\text{-}\mathsf{bqo}$ that appears below. In these cases, the equivalence appears to require a considerably stronger base theory. For our results in reverse mathematics, we work with the standard definition in terms of barriers that was given above. Most considerations will also apply to blocks.

It has been shown by A.~Montalb\'an~\cite{montalban-fraisse} that Fra\"iss\'e's conjecture can be proved in the theory~$\Pi^1_1\text{-}\mathsf{CA}_0$ from reverse mathematics. To establish this result, Montalb\'an gave a new proof of Fra\"iss\'e's conjecture, as Laver's argument relies on the minimal bad array principle, which had been conjectured to be unprovable in~$\Pi^1_1\text{-}\mathsf{CA}_0$. This last conjecture was recently confirmed by A.~Freund, F.~Pakhomov and G.~Sold\`a~\cite{bad-array-Pi^1_2-CA}, who proved that the minimal bad array principle is equivalent to the even stronger set existence principle of $\Pi^1_2$-comprehension, over the base theory~$\mathsf{ATR}_0$.

For $n\in\mathbb N$, we write $n$ for the usual linear order and $\overline n$ for the antichain on the underlying set~$\{0,\ldots,n-1\}$. An important ingredient for the aforementioned result by Montalb\'an is the statement that $\overline 3$ is $\Delta^0_2\text{-}\mathsf{bqo}$. The latter entails arithmetic transfinite recursion over~$\mathsf{RCA}_0$, as recently shown by Freund~\cite{freund-3-bqo}, who also showed that arithmetic recursion along~$\mathbb N$ (the central axiom of~$\mathsf{ACA}_0^+$) follows from the statement that $\overline 3$ is $\mathsf{bqo}$.

If an extension of $\mathsf{RCA}_0$ proves that $\overline 3$ is $\mathsf{bqo}$, it proves that the same holds for all finite quasi orders, by a result of Marcone \cite{marcone-survey-old} and its strengthening by Freund~\cite{freund-3-bqo}. This suggests the following question:
\begin{quote}
Which finite partial orders are provably $\mathsf{bqo}$ in theories like~$\mathsf{ACA}_0$, which do not prove that $\overline 3$ is $\mathsf{bqo}$?
\end{quote}
In this paper we answer this question and also give a characterization of the infinite partial orders that are provably~$\mathsf{bqo}$ in these theories. This connects with the classical topic of provable well orders, which is studied in ordinal analysis (see, e.\,g., \cite{rathjen-realm} for background). Let us note that the focus on partial orders is pure convenience, as any quasi order is equivalent to an anti-symmetric quotient.

To state our results, we introduce some notation. Given a partial order $(I,\preceq)$ and quasi orders $(Q_i,\leq_i)$ for $i\in I$, we consider the quasi order given by
\begin{gather*}
    \textstyle\sum_{i\in I}Q_i=\{(i,q)\,|\,i\in I\text{ and }q\in Q_i\},\\
    (i,q)\leq(j,r)\text{ in }\textstyle\sum_{i\in I}Q_i\quad\Leftrightarrow\quad i\prec j\text{ or }(i=j\text{ and }q\leq_i r).
\end{gather*}
Let us note that the binary sum with incomparable summands can be recovered as
\begin{equation*}
    Q_0\oplus Q_1=\textstyle\sum_{i\in\overline 2}Q_i.
\end{equation*}
By a linear sum, we mean an order $\sum_{i\in I}Q_i$ such that~$I$ is linear. When the latter is also well-founded, we speak of a well-ordered sum. We say that a quasi order is a linear or well-ordered sum of orders with a certain property if it is isomorphic to a linear or well-ordered sum $\sum_{i\in I}Q_i$ such that each $Q_i$ has the property in question. Let us also recall that a function $f:P\to Q$ between quasi orders is said to be order reflecting if $f(p)\leq_Q f(q)$ implies $p\leq_P q$ for all~$p,q\in P$. When all these implications are equivalences, $f$ is called an embedding. If $Q$ is $\mathsf{bqo}$ and there is an order reflecting map $f:P\to Q$, then $P$ is $\mathsf{bqo}$, since any bad array $g:B\to P$ would induce a bad array~$f\circ g:B\to Q$.

Marcone~\cite{marcone-survey-old} has shown that $\mathsf{RCA}_0$ proves $\overline 2$ to be~$\mathsf{bqo}$. For our characterization of provable $\mathsf{bqo}$s, we will combine this fact with the new result that $\mathsf{RCA}_0$ proves each of the following:
\begin{itemize}
    \item the order $1\oplus 2$ is $\mathsf{bqo}$ precisely if the same holds for~$\overline 3$ (Corollary~\ref{cor:one-two-three}),
    \item a partial order $P$ is a linear sum of antichains of size at most two precisely if there is no order reflecting map of $1\oplus 2$ into~$P$ (Corollary~\ref{cor:quasi-emb-1oplus2}),
    \item the class of $\mathsf{bqo}$s is closed under well-ordered sums (Proposition~\ref{prop:wo-sum-bqo}).
\end{itemize}
Concerning the third point, we note that the $\mathsf{bqo}$s are actually closed under sums where the index set is $\mathsf{bqo}$ (not necessarily linear). In view of $\sum_{p\in\overline 2}\overline 2\cong\overline 4$, however, this generalization requires a stronger base theory.

For extensions of $\mathsf{RCA}_0$ that do not prove~$\overline 3$ to be $\mathsf{bqo}$, the finite partial orders that are provably~$\mathsf{bqo}$ can thus be characterized as the linear sums of antichains $\overline 1$ and~$\overline 2$ (Theorem~\ref{thm:prov-bqo-fin}). By combining the above with classical results on provable well orders, we can also characterize the infinite partial orders that are provably~$\mathsf{bqo}$ in such axiom systems (Theorem~\ref{thm:prov-bqo-inf}).

As an application of our work on provable~$\mathsf{bqo}$s, we provide formal evidence that the aforementioned result on the minimal bad array principle~\cite{bad-array-Pi^1_2-CA} requires a reasonably strong base theory. Specifically, we show that a certain version of that principle does not entail arithmetic transfinite recursion over~$\mathsf{ACA}_0$ (Corollary~\ref{cor:MBA-ATR}). As an interesting counterpoint, we also show that another version of the principle entails arithmetic comprehension over~$\mathsf{RCA}_0$ (Proposition~\ref{prop:mba-to-aca}).

\section{One plus two is three}

Let us recall that we write $n$ for the linear order and $\overline n$ for the antichain with underlying set~$\{0,\ldots,n-1\}$. We use $\oplus$ for the disjoint union of orders in which the summands are incomparable. In the present section, we show that $1\oplus 2$ is a better quasi order precisely when the same holds for~$\overline 3$, provably in~$\mathsf{RCA}_0$. By a~previous result~\cite{freund-3-bqo}, it follows that $1\oplus 2$ being $\mathsf{bqo}$ entails at least $\mathsf{ACA}_0^+$ (arithmetic recursion along~$\mathbb N$). In the next section, we will see that $1\oplus 2$ and $\overline 3$ act as forbidden suborders for the partial orders that are provably $\mathsf{bqo}$ in moderately weak theories.

A quasi order~$Q$ is~$\mathsf{bqo}$ precisely when a suitable order on the hereditarily countable sets with urelements from~$Q$ is well-founded or equivalently~$\mathsf{bqo}$. To avoid confusion, we stress that well-foundedness is not equivalent to being $\mathsf{bqo}$ in general but only in the indicated case. The characterization of $\mathsf{bqo}$s in terms of sets with urelements can be traced back to the original work of Nash-Williams~\cite{nash-williams-trees} (see~\cite{pequignot-survey} for a detailed proof). As in~\cite{freund-3-bqo}, we focus on the hereditarily finite case.

\begin{definition}[$\mathsf{RCA}_0$]\label{def:H(Q)}
Let us consider a quasi order~$Q$. We recursively generate a set $H_f(Q)$ by the following clauses:
\begin{enumerate}[label=(\roman*)]
\item for each $q\in Q$ we include an element $(0,q)\in H_f(Q)$,
\item we add an element $(1,a)\in H_f(Q)$ for each finite set $a\subseteq H_f(Q)$ with elements that we have already constructed.
\end{enumerate}
Furthermore, we define a quasi order $\leq_{H(Q)}$ on the set $H_f(Q)$ by stipulating that the following clauses are satisfied:
\begingroup
\allowdisplaybreaks
\begin{align*}
(0,p)\leq_{H(Q)}(0,q)\quad&\Leftrightarrow\quad p\leq q\text{ holds in }Q,\\
(0,p)\leq_{H(Q)}(1,b)\quad&\Leftrightarrow\quad \text{there is some $y\in b$ with }(0,p)\leq_{H(Q)}y,\\
(1,a)\leq_{H(Q)}(0,q)\quad&\Leftrightarrow\quad \text{all $x\in a$ validate }x\leq_{H(Q)}(0,q),\\
(1,a)\leq_{H(Q)}(1,b)\quad&\Leftrightarrow\quad \text{each $x\in a$ admits a $y\in b$ with $x\leq_{H(Q)}y$}.
\end{align*}
\endgroup
To improve readability, we shall from now on write $q\in Q$ and $a\in H_f(Q)\backslash Q$ in order to refer to the elements $(0,q)$ and $(1,a)$ of the set $H_f(Q)$.
\end{definition}

In $\mathsf{RCA}_0$, the elements of $H_f(Q)$ are not to be represented by sets in the sense of second-order objects but rather by numerical codes for finite trees or terms with leaf labels or constant symbols from~$Q$. Correspondingly, we consider $\leq_{H(Q)}$ as a primitive recursive relation between these codes. An induction over trees or terms confirms the implicit claim that $H_f(Q)$ is a quasi order. The following result is Theorem~3.2 of the indicated reference.

\begin{proposition}[$\mathsf{RCA_0}$; \cite{freund-3-bqo}]\label{prop:H(Q)-bqo}
If $Q$ is $\mathsf{bqo}$, then so is $H_f(Q)$.
\end{proposition}

For $x\in H_f(Q)$ considered as a tree, we write $\operatorname{supp}(x)$ for the set of leaf labels of~$x$, which means that we have
\begin{align*}
    \operatorname{supp}(q)&=\{q\}&&\text{for $q\in Q$},\\
    \operatorname{supp}(a)&=\textstyle\bigcup\{\operatorname{supp}(x)\,|\,x\in a\}&&\text{for $a\in H_f(Q)\backslash Q$},
\end{align*}
so that $\operatorname{supp}(x)$ is a finite subset of~$Q$ and hence an element of~$H_f(Q)$. As in~\cite{freund-3-bqo}, a straightforward induction over trees or terms yields the following.

\begin{lemma}[$\mathsf{RCA_0}$; \cite{freund-3-bqo}]\label{lem:H-order-basic}
For any quasi order~$Q$ we have
\begin{align*}
    x\leq_{H(Q)}y\quad&\Rightarrow\quad\operatorname{supp}(x)\leq_{H(Q)}\operatorname{supp}(y) && \text{for any $x,y\in H_f(Q)$},\\
    x\in a\quad&\Rightarrow\quad x\leq_{H(Q)}a && \text{when $a\in H_f(Q)\backslash Q$}.
\end{align*}
\end{lemma}

The order $H_f(\overline 3)$ contains two independent copies of the natural numbers, as shown in~\cite{freund-3-bqo}. We now show that two `interlocked' copies can be found in~$H_f(1\oplus 2)$. Let us agree to write $1\oplus 2=\{\star\}\cup\{0,1\}$, where $0<1$ is the only strict inequality.

\begin{definition}[$\mathsf{RCA}_0$]
For $n\in\mathbb N$ we define $\dot n,\ddot n\in H_f(1\oplus 2)$ recursively by
\begin{equation*}
   \dot n=\{\star,0\}\cup\{\dot m\,|\,m<n\},\qquad\ddot n=\{\star,1\}\cup\{\ddot m\,|\,m<n\}.
\end{equation*}
\end{definition}

The following is a variation on a result about~$H_f(\overline 3)$ that was proved in~\cite{freund-3-bqo}. The difference is that $\dot m$ and $\ddot n$ are always incomparable in~$H_f(\overline 3)$, while we get $\dot n\leq\ddot n$ since $0<1$ holds in~$1\oplus 2$.

\begin{proposition}[$\mathsf{RCA}_0$]\label{prop:interlocked-copies}
    For any $m,n\in\mathbb N$ we have
    \begin{equation*}
        \dot m\leq_{H(1\oplus 2)}\dot n\quad\Leftrightarrow\quad
        \ddot m\leq_{H(1\oplus 2)}\ddot n\quad\Leftrightarrow\quad
        \dot m\leq_{H(1\oplus 2)}\ddot n\quad\Leftrightarrow\quad m\leq n.
    \end{equation*}
    Furthermore, we have $\ddot m\not\leq_{H(1\oplus 2)}\dot n$ for any~$m,n\in\mathbb N$.
\end{proposition}
\begin{proof}
    To improve readability, we write $\leq_H$ for the order relation on~$H_f(1\oplus 2)$. To establish the final claim of the proposition, we first note that we have
    \begin{equation*}
        \operatorname{supp}(\ddot m)=\{\star,1\}\quad\text{and}\quad\operatorname{supp}(\dot n)=\{\star,0\}.
    \end{equation*}
    Given that both $1\not\leq\star$ and $1\not\leq 0$ holds in $1\oplus 2$, we get $\operatorname{supp}(\ddot m)\not\leq_H\operatorname{supp}(\dot n)$. To conclude $\ddot m\not\leq_H\dot n$, we now invoke Lemma~\ref{lem:H-order-basic}. 
    
    In the following, we prove the last of the given equivalences. The proof of the other equivalences is similar and can be found in~\cite{freund-3-bqo}.
    
    To show that $m\leq n$ entails $\dot m\leq_H\ddot n$, we use induction on~$n$. Given $m\leq n$, the task is to show that each $x\in\dot m$ admits a $y\in\ddot n$ with $x\leq_H y$. For $x=\star$ and $x=0$ we can take $y=\star$ and $y=1$, respectively. In the remaining case we have $x=\dot k$ for some~$k<n$. The induction hypothesis ensures that $y=\ddot k$ is as required.
    
    We now use induction on~$m$ to show that $\dot m\leq_H\ddot n$ entails $m\leq n$ for all~$n$. Note that this amounts to a $\Pi^0_1$-induction, which is available in~$\mathsf{RCA}_0$. Aiming at a contradiction, we assume $\dot m\leq_H\ddot n$ but~$m>n$. The latter entails $\dot n\in\dot m$, so that we get $\dot n\leq_H y$ for some~$y\in\ddot n$. In view of $\operatorname{supp}(\dot n)\ni 0\not\leq\star$ and $\operatorname{supp}(\dot n)\ni\star\not\leq 1$, we can now use Lemma~\ref{lem:H-order-basic} to infer $y\notin\{\star,1\}$. The only other possibility is that we have $y=\ddot k$ for some~$k<n$. But then $\dot n\leq_H y$ contradicts the induction hypothesis.
\end{proof}

Let us now derive the promised result.

\begin{corollary}[$\mathsf{RCA}_0$]\label{cor:one-two-three}
The order~$1\oplus 2$ is $\mathsf{bqo}$ precisely if the same holds for~$\overline 3$.
\end{corollary}
\begin{proof}
    The backward direction is immediate, since there is an order reflecting map from $1\oplus 2$ into~$\overline 3$. To prove the forward direction via Proposition~\ref{prop:H(Q)-bqo}, we show that $H_f(1\oplus 2)$ contains an antichain of size three. Such an antichain is given by
    \begin{equation*}
        \{\ddot 0,\dot 5\},\quad\{\ddot 1,\dot 4\},\quad\{\ddot 2,\dot 3\}.
    \end{equation*}
    As a representative case, we explain why the first two elements are incomparable. Once again, we write $\leq_H$ for the inequality on~$H_f(1\oplus 2)$. Due to Proposition~\ref{prop:interlocked-copies} we have $\dot 5\not\leq_H\ddot 1$ and $\dot 5\not\leq_H\dot 4$, so that $\{\ddot 0,\dot 5\}\not\leq_H\{\ddot 1,\dot 4\}$ follows by Definition~\ref{def:H(Q)}. We also have $\ddot 1\not\leq_H\ddot 0$ and $\ddot 1\not\leq_H\dot 5$, which yields $\{\ddot 1,\dot 4\}\not\leq_H\{\ddot 0,\dot 5\}$.
\end{proof}

By the main result of~\cite{freund-3-bqo}, we can conclude the following.

\begin{corollary}\label{cor:1+2-bqo-strong}
    Arithmetic recursion along~$\mathbb N$ ($\mathsf{ACA}_0^+$) follows from the statement that $1\oplus 2$ is $\mathsf{bqo}$, over~$\mathsf{RCA}_0$. In particular, $\mathsf{ACA}_0$ cannot prove that $1\oplus 2$ is $\mathsf{bqo}$.
\end{corollary}

\section{A characterization of provable better quasi orders}

Our aim in the present section is to characterize those partial orders that are provably $\mathsf{bqo}$ when~$\overline 3$ is not. In view of Corollary~\ref{cor:one-two-three}, the following forbidden minor characterization will play a central role. The precise meaning of statement~(i) in the following result has been explained in the introduction.

\begin{proposition}[$\mathsf{RCA}_0$]\label{prop:1oplus2-sum-of-antichains}
For any partial order~$P$, the following are equivalent:
\begin{enumerate}[label=(\roman*)]
\item the order~$P$ is a linear sum of antichains,
\item the reflexive closure of incomparability in~$P$ is an equivalence relation,
\item there is no order embedding of~$1\oplus 2$ into~$P$.
\end{enumerate}
\end{proposition}
\begin{proof}
Clearly (i) implies~(ii). Since the reflexive closure of incomparability is not transitive in $1\oplus 2$, it is also not transitive in any order into which $1\oplus 2$  can be embedded, so that~(ii) implies (iii). 

To see that (iii) implies (ii), notice that any $P$ where the reflexive closure of incomparability is not transitive (which is the only reason why it could fail to be an equivalence relation) will have three distinct elements $x,y,z$ such that $x$ is incomparable with $y$, the latter is incomparable with $z$, but $x$ and $z$ are comparable. This clearly gives us an embedding of $1\oplus 2$ into $P$ (where the element of $1$ is mapped to $y$ and those of $2$ are mapped to $x$ and $z$).

Finally, we assume that an order $P$ satisfies (ii) and show that it also satisfies~(i). Clearly, equivalence classes with respect to the reflexive closure of incomparability are antichains. At the same time, the order $P$ is compatible with this equivalence relation: Indeed, if we have incomparable $x_1,x_2$ and some other element $y$ that is not incomparable with them, then $P$-comparisons of $x_1$ with $y$ and of $x_2$ with $y$ will agree, since otherwise transitivity would yield either $x_1<_P x_2$ or $x_1>_P x_2$. Let~$C$ be the quotient of $P$ by the reflexive closure of the incomparability relation. As~usual in reverse mathematics, we work with orders that are relations on subsets of~$\mathbb{N}$. For later reference, we officially define $C$ as the suborder of $P$ that contains the $<_\mathbb{N}$-smallest element of each equivalence class. Let us write $A(p)$ for the equivalence class that contains~$p\in C$. Then we have $P\cong\sum_{p\in C}A(p)$, as needed for (i).
\end{proof}

As recalled in the introduction, Marcone~\cite{marcone-survey-old} has shown that $\mathsf{RCA}_0$ proves the statement that $\overline 2$ is~$\mathsf{bqo}$. This explains the relevance of the following result.

\begin{corollary}[$\mathsf{RCA}_0$]\label{cor:quasi-emb-1oplus2}
For any partial order~$P$, the following are equivalent:
\begin{enumerate}[label=(\roman*)]
\item the order~$P$ is a linear sum of antichains with at most two elements each,
\item there is no order reflecting map from~$1\oplus 2$ into~$P$,
\item there is no embedding of~$1\oplus 2$ into~$P$ and no embedding of $\overline 3$ into~$P$.
\end{enumerate}
\end{corollary}
\begin{proof}
Given the previous proposition, it suffices to note that a map from~$1\oplus 2$ into~$P$ is order reflecting but not an embedding precisely when the image is an antichain with three elements.
\end{proof}

To characterize the provable $\mathsf{bqo}$s of certain theories, we will combine the previous corollary with the following result. As indicated in the introduction, the corresponding result for $\mathsf{bqo}$-indexed sums can be established in the much stronger theory~$\mathsf{ATR}_0$ (form subarrays that are perfect with respect to the indexing order). The base theory $\mathsf{RCA}_0$ cannot suffice for this generalization, since it proves that $\overline 2$ is $\mathsf{bqo}$ but does not prove that the same holds for $\overline 4\cong\sum_{p\in\overline 2}\overline 2$.

\begin{proposition}[$\mathsf{RCA}_0$]\label{prop:wo-sum-bqo}
Any well-ordered sum of $\mathsf{bqo}$s is itself $\mathsf{bqo}$.
\end{proposition}
\begin{proof}
    Let us consider a bad array $f:B\to\sum_{i\in I}Q_i$ into a well-ordered sum. We write $f_0:B\to I$ for the composition of~$f$ with the map $\sum_{i\in I}Q_i\ni(i,q)\mapsto i\in I$. Given $s\in B$, we define $B/s$ as the subbarrier that consists of all $t\in B$ that have minimal element strictly above the maximal element of $s$. For $t\in B/s$ we find intervals $r^i\in B$ of $s\cup t$ that form a chain $s=r^0\vartriangleleft\ldots\vartriangleleft r^n=t$. Given that $f$ is bad, we must have $f_0(s)=f_0(r^0) \ge_I \ldots \ge_I f_0(r^n)=f_0(t)$. 

    Suppose that for all $s \in B$ there exists $t \in B/s$ with $f_0(s) >_I f_0(t)$. Then, by choosing the least $t$ (with respect to $\leq_{\mathbb N}$), we can define a descending sequence in~$I$, against the fact that $I$ is well-founded. Hence there must be an element $s\in B$ such that $f_0$ assumes the constant value~$i=f_0(s)\in I$ on~$B/s$. But then $f_1:B/s\to Q_i$ with $f(t)=(i,f_1(t))$ is bad, so that $Q_i$ cannot have been a better quasi order.
\end{proof}

We now derive a characterization of provable better quasi orders. In the infinite case, we consider partial orders as living in the \lq\lq real world\rq\rq\ and being represented in subsystems of second-order arithmetic by appropriate descriptions, so that intensional aspects play a role (consider $\mathbb N$ as a linear order if some large cardinal notion is consistent and as an antichain otherwise). Furthermore, since $\mathsf{RCA}_0$ proves that a linear order is $\mathsf{bqo}$ precisely if it is well-founded, the infinite case relates to proof-theoretic ordinals. For these reasons, we begin with the more straightforward case of finite orders, which we assume to be represented by a fixed standard system of numerical codes (say via incidence matrices). If a finite order has a property that is given by a $\Sigma^0_1$-condition on the code (can be established by a finite verification), then this fact from the \lq\lq real world\rq\rq\ can already be proved in $\mathsf{RCA}_0$, by the principle of $\Sigma^0_1$-completeness (see, e.\,g., Theorem~I.1.8 of~\cite{hajek91}).

Let us note that $\mathsf{ACA}_0$ satisfies the conditions on~$\mathsf T$ in the following result; also, if some theory $\mathsf T\supseteq\mathsf{RCA}_0$ proves that $\overline 3$ is bqo, it proves the same for any finite partial order (see Corollaries~2.13 and 3.9 of~\cite{freund-3-bqo}).

\begin{theorem}\label{thm:prov-bqo-fin}
Consider a theory $\mathsf T$ in the language of second order arithmetic that extends $\mathsf{RCA}_0$ and does not prove that the antichain with three elements is a better quasi order. For any finite partial order~$P$, the following are equivalent:
\begin{enumerate}[label=(\roman*)]
\item the theory $\mathsf T$ proves that $P$ is a better quasi order,
\item the order~$P$ is a linear sum of antichains with at most two elements each.
\end{enumerate}
\end{theorem}
\begin{proof}
    First assume that~(ii) fails. Then Corollary~\ref{cor:quasi-emb-1oplus2} ensures that there is an order reflecting map from $1\oplus 2$ into~$P$. This fact is recognized by~$\mathsf T$ due to the principle of $\Sigma^0_1$-completeness. So $\mathsf T$ proves that $1\oplus 2$ is $\mathsf{bqo}$ if the same holds for~$P$. In view of Corollary~\ref{cor:one-two-three}, we can conclude that~(i) fails unless $\mathsf T$ proves that $\overline 3$ is $\mathsf{bqo}$.

    Let us now assume that (ii) holds. As we are concerned with finite orders, this fact is recognized by~$\mathsf T$ due to $\Sigma^0_1$-completeness, and a linear sum is the same as a well-ordered sum. Also, $\mathsf T$ knows that antichains of size at most two are~$\mathsf{bqo}$, due to Lemma~3.2 of~\cite{marcone-survey-old}. We get~(i) by Proposition~\ref{prop:wo-sum-bqo} above.
\end{proof}

As in the case of Theorem~\ref{thm:prov-bqo-fin}, the following result extends to other theories that do not prove that $\overline 3$ is~$\mathsf{bqo}$. We formulate the result for a specific case in order to avoid a general discussion of standard notation systems for proof-theoretic ordinals. The order $\overline 2\cdot\gamma$ is defined to be $\sum_{\beta\in\gamma}P_\beta$ with $\gamma=\{\beta\in\varepsilon_0\,|\,\beta<\gamma\}\subseteq\varepsilon_0$ and $P_\beta=\overline 2$ for all $\beta\in\gamma$. In other words, it is the lexicographic order on~$\gamma\times\overline 2$, which corresponds to the usual ordinal arithmetic (note the reverse order of factors). 

\begin{theorem}\label{thm:prov-bqo-inf}
For any partial order~$Q$, the following are equivalent:
\begin{enumerate}[label=(\roman*)]
\item the order $Q$ is isomorphic to a computably enumerable suborder of~$\overline 2\cdot\gamma$ for some $\gamma<\varepsilon_0$ (where $\varepsilon_0$ is represented by standard notations as in~\mbox{\cite[\S\,11]{takeuti-book}}),
\item there is a computable presentation of an order~$P\cong Q$ for which $\mathsf{ACA}_0$ shows that $P$ is a better quasi order.
\end{enumerate}
\end{theorem}
\begin{proof}
To see that~(i) implies~(ii), recall that $\mathsf{ACA}_0$  has proof-theoretic ordinal~$\varepsilon_0$. So for each $\gamma<\varepsilon_0$ that is fixed externally, $\mathsf{ACA}_0$ proves that $\gamma$ is well-founded (see, e.\,g.,~\cite[\S\,13]{takeuti-book} for this result due to G.~Gentzen~\cite{gentzen38,gentzen43}). By Proposition~\ref{prop:wo-sum-bqo} above and Lemma~3.2 of~\cite{marcone-survey-old}, it follows that $\mathsf{ACA}_0$ proves $\overline 2\cdot\gamma=\sum_{\beta<\gamma}\overline 2$ to be $\mathsf{bqo}$.

Now assume $Q$ is isomorphic to the image of a computable function $e:\mathbb N\to\overline 2\cdot\gamma$. Let $P$ be the order on pairs $(p,n)\in(\overline 2\cdot\gamma)\times\mathbb N$ such that $n$ is minimal with~$e(n)=p$, where $(p,m)\leq(q,n)$ holds in $P$ precisely if we have $p\leq q$ in $\overline 2\cdot\gamma$. Then $P$ is computable and isomorphic to $Q$ via the projection $(p,n)\mapsto p$. In $\mathsf{ACA}_0$ we know that each~$p$ admits at most one~$n$ with $(p,n)\in P$ (by definition of~$P$), which entails that $P$ is a partial order and that the projection is an embedding into~$\overline 2\cdot\gamma$. By the above, the fact that~$P$ is $\mathsf{bqo}$ can thus be proved in~$\mathsf{ACA}_0$.

For the converse implication, we consider a computable presentation of a partial order~$P$ as in~(ii). We may assume that $\mathsf{ACA}_0$ proves $P$ to be a partial order (rather than just a quasi order), by considering another computable description that picks minimal codes among equivalent elements. Let $C\subseteq P$ be defined as in the proof of Proposition~\ref{prop:1oplus2-sum-of-antichains}. We later show that $P$ is isomorphic to $\sum_{p\in C}A(p)$ as in that~proof, but this fact may not be available in~$\mathsf{ACA}_0$ (which may not know that (i) holds). However, the latter does recognize that $C$ is a linear suborder, which must be a well order when~$P$ is~$\mathsf{bqo}$. Due to the ordinal analysis of Gentzen (see Theorem~13.4 of~\cite{takeuti-book}), we thus get a computable embedding $f:C\to\gamma$ for some~$\gamma<\varepsilon_0$.

We now show that there can be no order reflecting map from $1\oplus 2$ into~$P$. If there was, $\mathsf{ACA}_0$ would recognize this, due to $\Sigma^0_1$-completeness. Given that $\mathsf{ACA}_0$ proves $P$ to be $\mathsf{bqo}$, it would prove the same for $1\oplus 2$, against Corollary~\ref{cor:1+2-bqo-strong}. By the proofs of Proposition~\ref{prop:1oplus2-sum-of-antichains} and Corollary~\ref{cor:quasi-emb-1oplus2}, it follows that $P$ is isomorphic to a certain sum $\sum_{p\in C}A(p)$ of antichains $A(p)\ni p$ with at most two elements. Using~$f:C\to\gamma$ from above, we obtain an embedding $g:\sum_{p\in C}A(p)\to\overline 2\cdot\gamma$ that is given by $g(p)=(f(p),0)$ for $p\in C$ and $g(q)=(f(p),1)$ for $q\in A(p)\backslash\{p\}$ (recall that the order $\overline 2\cdot\gamma$ has underlying set $\gamma\times\overline 2$). The image of~$g$ is the computably enumerable suborder required by~(i).
\end{proof}

Let us point out that statement~(ii) of Theorem~\ref{thm:prov-bqo-inf} is parallel to classical characterizations of provable well orders. In particular, these also involve the choice of a suitable presentation, given that the usual order on~$\mathbb N$ admits a non-standard description that looks ill-founded unless some strong consistency statement is valid (see, e.\,g., \cite[Section~2.1]{rathjen-realm} for this observation by G.~Kreisel). In contrast, statement~(i) of Theorem~\ref{thm:prov-bqo-inf} is slightly less straightforward than in the case of linear orders. This is because any suborder of~$\varepsilon_0$ is isomorphic to an initial segment and hence to a computable suborder. Since computability is automatic in this sense, the classical results on provable well orders can avoid reference to a standard notation system. In the case of~$\overline 2\cdot\gamma$, the components from $\overline 2$ may encode non-computable information when the components from $\gamma$ are collapsed onto an initial segment, so that computability is not automatic in the same sense. To obtain a version of our result that does not involve a choice of ordinal notations, one could consider orders that are $\Delta^1_1$ rather than computable.

\section{Minimal bad arrays over weak base theories}\label{sect:mba}

It was recently shown by Freund, Pakhomov and Sold\`a~\cite{bad-array-Pi^1_2-CA} that different versions of the minimal bad array principle are equivalent to $\Pi^1_2$-comprehension over~$\mathsf{ATR}_0$. Here we prove that one such version does not entail $\mathsf{ATR}_0$ over~$\mathsf{ACA}_0$. We also show that another version entails at least $\mathsf{ACA}_0$ over~$\mathsf{RCA}_0$.

According to the introduction, a $Q$-array is a function $f:B\to Q$ on a barrier~$B$. Such a function induces a map $F:[\bigcup B]^\omega\to Q$ with $F(X)=f(t)$ for $B\ni t\sqsubset X$. It is straightforward to see that $f$ is bad precisely if we have $F(X)\not\leq_Q F(X^-)$ for all $X\in[\bigcup B]^\omega$. Let us note that $F$ is continuous in the sense that each~$X\in[\bigcup B]^\omega$ admits an $s\sqsubset X$ such that $F$ is constant on~$\{Y\in[\bigcup B]^\omega\,|\,s\sqsubset Y\}$. Conversely, any map $F:[V]^\omega\to Q$ with $V\in[\mathbb N]^\omega$ that is continuous in this sense is induced by a function $f_0:B_0\to Q$ on a block with base $\bigcup B_0=V$. Here~$B_0$ may not be a barrier. At the same time, it is known from~\cite{cholak-RM-wpo} that the rather weak theory~$\mathsf{WKL}_0$ supports the construction of a barrier~$B$ with $\bigcup B\subseteq\bigcup B_0$ such that each $t\in B$ admits an $s\in B_0$ with~$s\sqsubseteq t$. The continuous map that is induced by $f:B\to Q$ with $f(t)=f_0(s)$ for $B_0\ni s\sqsubseteq t$ is a restriction of the map $F$ that we started with.

In the following, we assume that any continuous $F:[V]^\omega\to Q$ with $V\in[\mathbb N]^\omega$ is given as a function on a block that induces it. By a slight abuse of terminology, such an $F$ will also be called a $Q$-array. We say that it is bad when $F(X)\not\leq_Q F(X^-)$ holds for all~$X\in[V]^\omega$. As we have seen, the theory $\mathsf{WKL}_0$ ensures that $Q$ is $\mathsf{bqo}$ precisely when there is no bad $Q$-array in this new sense. In the absence of~$\mathsf{WKL}_0$, we insist on the previous definition of $\mathsf{bqo}$s in terms of arrays on barriers. We now introduce some notions that will occur in the minimal bad array principle.

\begin{definition}[$\mathsf{RCA}_0$]
    A partial ranking of a quasi order~$Q$ is a well-founded partial order~$\leq'$ on~$Q$ such that $p\leq'q$ entails $p\leq_Q q$. Given such a ranking, we write $F\leq' G$ for $Q$-arrays $F:[V]^\omega\to Q$ and $G:[W]^\omega\to Q$ if we have $V\subseteq W$ and $F(X)\leq'G(X)$ for all~$X\in[V]^\omega$. If we even have $F(X)<'G(X)$ for all such~$X$, then we write $F<'G$. By a $\leq'$-minimal bad $Q$-array we mean a bad $Q$-array~$G$ that admits no bad $Q$-array $F<'G$.
\end{definition}

Let us note that each well-founded partial order is a partial ranking of itself. We consider the following versions of the minimal bad array principle, which refer to arrays in the sense of continuous functions $[V]^\omega\to Q$ with $V\in[\mathbb N]^\omega$.
\begin{alignat*}{3}
&(\mathsf{MBA})\quad&& \parbox[t]{.7\textwidth}{When $\leq'$ is a partial ranking of a quasi order~$Q$, any bad $Q$-array $F_0$ admits a $\leq'$-minimal bad $Q$-array $F\leq' F_0$.}\\
&(\mathsf{MBA}^-)\quad&& \parbox[t]{.7\textwidth}{For each well-founded partial order~$Q$ that is no~$\mathsf{bqo}$, there is a $\leq_Q$-minimal bad $Q$-array.}
\end{alignat*}
In~\cite{bad-array-Pi^1_2-CA}, the principle $\mathsf{MBA}$ has been studied as `Simpson's version of the minimal bad array lemma'. Its formulation does indeed coincide with one that was given by Simpson~\cite{simpson-borel-bqos}, except that the latter works with a larger class of Borel measurable rather than continuous arrays (see the introduction of our paper). In~\cite{bad-array-Pi^1_2-CA}, the following equivalence includes yet another version of the minimal bad array principle, which goes back to work of Nash-Williams and has been isolated by Laver~\cite{laver-min-array}.

\begin{theorem}[$\mathsf{ATR}_0$;~\cite{bad-array-Pi^1_2-CA}]
    Each of $\mathsf{MBA}$ and $\mathsf{MBA}^-$ is equivalent to the strong set existence principle of $\Pi^1_2$-comprehension.
\end{theorem}
\begin{proof}
For $\mathsf{MBA}$, the equivalence holds by Theorem~1.3 of~\cite{bad-array-Pi^1_2-CA}. To obtain the~remaining equivalence, it suffices to note that the proof of Corollary~2.3 in~\cite{bad-array-Pi^1_2-CA} does only use $\mathsf{MBA}^-$ rather than $\mathsf{MBA}$.
\end{proof}

In the following, we show that the base theory $\mathsf{ATR}_0$ is necessary in the sense that $\mathsf{MBA}^-$ is weak over~$\mathsf{ACA}_0$. We begin with some preparations. Given a~partial order~$Q$ and some~$n\in\mathbb N$, let $[Q]^{\leq n}$ be the set of non-empty subsets of~$Q$ that have at most $n$ elements. For $a,b\in[Q]^{\leq n}$ we stipulate
\begin{equation*}
a\prec b\quad\Leftrightarrow\quad\text{each $p\in a$ admits a $q\in b$ with $p<_Q q$},
\end{equation*}
which defines a (strict) partial order~$\prec$ on~$[Q]^{\leq n}$.

\begin{lemma}
For each $n\in\mathbb N$, the theory $\mathsf{RCA}_0$ proves that the relation $\prec$ on~$[Q]^{\leq n}$ is well-founded whenever $Q$ is a well-founded partial order.
\end{lemma}
\begin{proof}
We argue by induction on~$n$ (external to~$\mathsf{RCA}_0$). The base case of $n=0$ is trivial. In the induction step, we derive a contradiction from the assumption that $a_0\succ a_1\succ\ldots$ is infinitely descending in~$[Q]^{\leq n+1}$. For each~$i\in\mathbb N$, let $b_i\subseteq a_i$ be minimal with $b_i\succ a_{i+1}$. First assume there is an infinite $I\subseteq\mathbb N$ with $b_i\neq a_i$ for all~$i\in I$. For $i<j$ in~$I$ we get $b_j\preceq a_j\preceq a_{i+1}\prec b_i$. So the $b_i$ form an infinitely descending sequence in~$[Q]^{\leq n}$, against the induction hypothesis. In the remaining case, we may assume that we have $b_i=a_i$ for all~$i\in\mathbb N$ (after passing to a tail of the original sequence). It follows that each $q\in a_i$ admits a $p\in a_{i+1}$ with $p<_Qq$. If this was false for~$q$, then $a_{i+1}\prec a_i$ would upgrade to $a_{i+1}\prec a_i\backslash\{q\}$, against the minimality of $b_i=a_i$. Now we can recursively pick $q_i\in a_i$ such that $q_0,q_1,\ldots$ is infinitely descending in~$Q$. But the latter was assumed to be well-founded.
\end{proof}

The following result will be central for our proof that $\mathsf{MBA}^-$ is weak over~$\mathsf{ACA}_0$. It involves an assumption that is false but consistent in view of Corollary~\ref{cor:1+2-bqo-strong}.

\begin{proposition}[$\mathsf{RCA}_0$]
    If $1\oplus 2$ is no $\mathsf{bqo}$, then $\mathsf{MBA}^-$ holds.
\end{proposition}
\begin{proof}
   We consider a well-founded partial order~$Q$ that is not~$\mathsf{bqo}$. There must be an order reflecting map from $1\oplus 2$ into~$Q$. If not, Corollary~\ref{cor:quasi-emb-1oplus2} would allow us to write~$Q\cong\sum_{i\in I}A_i$ for a linear order~$I$ and antichains~$A_i$ of size at most two. Given that~$Q$ is well-founded, the same would hold for~$I$. But then~$Q$ would be~$\mathsf{bqo}$ by Proposition~\ref{prop:wo-sum-bqo} and the fact that $\overline 2$ is~$\mathsf{bqo}$.
   
   By a bad triple we shall mean a subset of~$Q$ that constitutes the range of some function $1\oplus 2\to Q$ that is order reflecting (and hence in particular injective). We view bad triples as elements of $[Q]^{\leq 3}$ with the order~$\prec$ from the previous lemma, which is well-founded since the same holds for~$Q$. Let us consider a bad triple~$b\subseteq Q$ that is $\prec$-minimal. By the assumption that $1\oplus 2$ is no better quasi order, we get a bad array $G:[W]^\omega\to Q$ with range~$b$. To establish~$\mathsf{MBA}^-$ by contradiction, we assume that $F<_Q G$ holds for some bad array~$F:[V]^\omega\to Q$. We may view $F$ as an array into the suborder
\begin{equation*}
Q_0=\{p\in Q\,|\,p<_Q q\text{ for some }q\in b\}.
\end{equation*}
As $F$ witnesses that $Q_0$ is no~$\mathsf{bqo}$, the latter must contain a bad triple~$a$. But we have $a\prec b$ by definition of~$Q_0$, against the minimality of~$b$.
\end{proof}

Let us now give the promised application. As the following proof reveals, the result remains valid when $\mathsf{ACA}_0$ is replaced by some other theory that does not prove~$1\oplus 2$ to be~$\mathsf{bqo}$, while arithmetic transfinite recursion could be replaced by the possibly weaker statement that~$1\oplus 2$ is~$\mathsf{bqo}$.

\begin{corollary}\label{cor:MBA-ATR}
    In the theory $\mathsf{ACA}_0$ one cannot prove that $\mathsf{MBA}^-$ entails the principle of arithmetic transfinite recursion~($\mathsf{ATR}_0$).
\end{corollary}
\begin{proof}
    Let~$\mathsf T$ be the extension of $\mathsf{ACA}_0$ by the statement that~$1\oplus 2$ is no~$\mathsf{bqo}$. This theory is consistent by Corollary~\ref{cor:1+2-bqo-strong}. The previous proposition tells us that $\mathsf T$ proves~$\mathsf{MBA}^-$. So if the present claim was false, then $\mathsf T$ would prove arithmetic transfinite recursion. But the latter entails that $1\oplus 2$ is~$\mathsf{bqo}$, by an application of the clopen Ramsey theorem (see~\cite{marcone-survey-old}). So~$\mathsf T$ would be inconsistent.
\end{proof}

We conclude this paper with the following counterpoint to the previous result. Let us note that we do not know whether either result remain valid when $\mathsf{MBA}^-$ is replaced by $\mathsf{MBA}$ and vice versa. Also, we do not know whether the following can be extended beyond arithmetic comprehension. Finally, we point out that the following proof makes no use of one particular feature of~$\mathsf{MBA}$, namely, that a minimal bad array can be found below a given bad array.

\begin{proposition}[$\mathsf{RCA}_0$]\label{prop:mba-to-aca}
Arithmetic comprehension ($\mathsf{ACA}_0$) follows from~$\mathsf{MBA}$.
\end{proposition}
\begin{proof}
Given a linear order~$\alpha$, we write $\omega^\alpha$ for the set of finite sequences in~$\alpha$ that are weakly decreasing, ordered lexicographically. The principle that $\omega^\alpha$ is well-founded for any well order~$\alpha$ is equivalent to arithmetic comprehension over~$\mathsf{RCA_0}$, as shown by J.-Y.~Girard~\cite{girard87} and J.~Hirst~\cite{hirst94}.

Let us write $\sigma=\langle \sigma_0,\ldots,\sigma_{l(\sigma)-1}\rangle$ to refer to the length and entries of a finite sequence. For $\sigma,\tau\in \omega^\alpha$ we now stipulate
\begin{equation*}
\sigma\leq'\tau\quad\Leftrightarrow\quad \sigma=\langle\tau_i,\ldots,\tau_{l(\tau)-1}\rangle\text{ for some }i\leq l(\tau).
\end{equation*}
This yields a partial ranking of the order on~$\omega^\alpha$. Towards a contradiction, we assume that $\alpha$ is a well order while $\sigma^0,\sigma^1,\ldots\subseteq\omega^\alpha$ is strictly decreasing. We get a bad array $F_0:[\mathbb N]^\omega\to\omega^\alpha$ by stipulating that we have $F_0(X)=\sigma^x$ when $x$ is the minimal element of~$X$.

We now invoke $\mathsf{MBA}$ to obtain a $\leq'$-minimal bad array $G:[W]^\omega\to\omega^\alpha$. As explained above, the latter is represented by a function $g:B\to\omega^\alpha$ on a block, in the sense that we have $G(X)=g(t)$ for $B\ni t\sqsubset X$. Given that $G$ and hence~$g$ is bad, we have $g(s)\not\leq g(t)$ in~$\omega^\alpha$ for any $s,t\in B$ with $s\vartriangleleft t$. In particular, each sequence $g(s)$ is nonempty. Let $g(s)_\star=\langle g(s)_1,\ldots,g(s)_{l(g(s))-1}\rangle$ denote the sequence that results from $g(s)$ when the first entry $g(s)_0$ is removed. Given $s\vartriangleleft t$, we get $g(s)_0\geq g(t)_0$ in~$\alpha$, and for $g(s)_0=g(t)_0$ we get $g(s)_\star\not\leq g(t)_\star$ in~$\omega^\alpha$.

As in the proof of Proposition~\ref{prop:wo-sum-bqo} (except that we now work with blocks rather than barriers), each $s\in B$ gives rise to a new block $B/s=\{t\in B\,|\,s_{l(s)-1}<t_0\}$. For any $t\in B/s$ we find $r^i\in B$ with $s=r^0\vartriangleleft\ldots\vartriangleleft r^n=t$, so that the above yields an inequality $g(s)_0\geq g(t)_0$ in~$\alpha$. Since the latter is well-founded, we may thus fix an $r\in B$ such that $g(t)_0=g(r)_0$ is constant for~$t\in B/r$. The point is that we have found a perfect $\alpha$-array without using the clopen Ramsey theorem, which is far beyond the reach of~$\mathsf{RCA_0}$.

For the $r\in B$ that we have just fixed, we now consider the function $f:B/r\to\omega^\alpha$ with $f(t)=g(t)_\star$. Given $s,t\in B/r$, we have $g(s)_0=g(r)_0=g(t)_0$, so that $s\vartriangleleft t$ entails $f(s)=g(s)_\star\not\leq g(t)_\star=f(t)$ in~$\omega^\alpha$, as seen above. Now let $F:[V]^\omega\to\omega^\alpha$ with $V=\bigcup B/r$ be given by $F(X)=f(t)$ for~$B/r\ni t\sqsubset X$. Then~$F$ is a bad array. For an arbitrary $X\in[V]^\omega$, we pick $t\in B/r\subseteq B$ with $t\sqsubset X$ to get
\begin{equation*}
    F(X)=f(t)=g(t)_\star<'g(t)=G(X).
\end{equation*}
But then we have~$F<'G$, against the assumption that~$G$ is minimal.
\end{proof}

\nocite{gentzen69}
\bibliographystyle{amsplain}
\bibliography{Provable-bqos}

\end{document}